\documentclass[10pt]{amsart}
\usepackage{amssymb}
\usepackage{amscd}
\usepackage[all]{xy}
\usepackage[colorlinks=true, linkcolor=red,
urlcolor=blue]{hyperref}

\numberwithin{equation}{section}

\def\today{\number\day\space\ifcase\month\or
 January\or February\or
   March\or April\or May\or June\or
    July\or August\or September\or
   October\or November\or December\fi\
     \number\year}

\theoremstyle{definition}
\newtheorem{thm}{Theorem}[section]
\newtheorem{lem}[thm]{Lemma}
\newtheorem{prp}[thm]{Proposition}
\newtheorem{dfn}[thm]{Definition}
\newtheorem{cor}[thm]{Corollary}
\newtheorem{cnj}[thm]{Conjecture}

\newtheorem{rmk}[thm]{Remark}
\newtheorem{ntn}[thm]{Notation}
\newtheorem{eme}[thm]{Example}

\newtheorem{qst}[thm]{Question}

\newcommand{\beq}{\begin{equation}}
\newcommand{\eeq}{\end{equation}}
\newcommand{\beqr}{\begin{eqnarray*}}
\newcommand{\eeqr}{\end{eqnarray*}}
\newcommand{\bal}{\begin{align*}}
\newcommand{\eal}{\end{align*}}
\newcommand{\bei}{\begin{itemize}}
\newcommand{\eei}{\end{itemize}}
\newcommand{\limi}[1]{\lim_{{#1} \to \infty}}

\newcommand{\dt}{\delta}
\newcommand{\ep}{\varepsilon}

\newcommand{\ta}{\tau}

\newcommand{\Q}{{\mathbb{Q}}}

\newcommand{\N}{{\mathbb{Z}}_{> 0}}

\newcommand{\Cu}{{\operatorname{Cu}}}
\newcommand{\Cone}{{\operatorname{C}}}
\newcommand{\F}{{\operatorname{F}}}

\newcommand{\Tr}{{\operatorname{Tr}}}

\newcommand{\Prim}{{\operatorname{Prim}}}
\newcommand{\diag}{{\operatorname{diag}}}

\newcommand{\rank}{{\operatorname{rank}}}

\newcommand{\drr}{{\operatorname{drr}}}
\newcommand{\QT}{{\operatorname{QT}}}
\newcommand{\T}{{\operatorname{T}}}
\newcommand{\W}{{\operatorname{W}}}

\newcommand{\Proj}{{\operatorname{Proj}}}

\newcommand{\rc}{{\operatorname{rc}}}

\newcommand{\cH}{{\mathcal{H}}}

\newcommand{\cK}{{\mathcal{K}}}

\newcommand{\dirlim}{\varinjlim}
\newcommand{\invlim}{\varprojlim}
\newcommand{\Mi}{M_{\infty}}

\newcommand{\tfae}{the following are equivalent}

\newcommand{\ca}{C*-algebra}

\newcommand{\I}{\infty}

\newcommand{\Lem}[1]{Lemma~\ref{#1}}
\newcommand{\Def}[1]{Definition~\ref{#1}}

\newcommand{\Prp}[1]{Proposition~\ref{#1}}

\newcommand{\Rmk}[1]{Remark~\ref{#1}}

\title[The radius of comparison of the
tensor product]{The radius of comparison of the
tensor product of a C*-algebra with $C (X)$}   

\author{Mohammad B. Asadi}

\address{School of Mathematics, Statistics and Computer Science,
College of Science, University of Tehran, Tehran, Iran.}

\email[]{mb.asadi@ut.ac.ir}

\author{M.~Ali Asadi-Vasfi}

\address{School of Mathematics, Statistics and Computer Science,
College of Science, University of Tehran, Tehran, Iran.}

\email[]{Asadi.ali@ut.ac.ir}

\date{\today}


\begin{document}
\begin{abstract}
Let $X$ be a compact metric space, let $A$ be  a unital AH algebra
with large matrix sizes, and let $B$ be  a stably finite  unital
\ca.
 Then
we give a lower bound for the radius of comparison of $C(X) \otimes B$ and
prove that
  the dimension-rank ratio satisfies
$\drr (A) = \drr \left(C(X)\otimes A\right)$.
We also give a class of unital AH algebras $A$ with $\rc \left(C(X) \otimes A\right) = \rc (A)$.
We further give a class of stably finite exact $\mathcal{Z}$-stable
 unital C*-algebras with nonzero radius of comparison.
\end{abstract}
\maketitle

\section{Introduction}
The radius of comparison of a C*-algebra, based on the Cuntz semigroup, and the dimension-rank ratio of an AH algebra
are numerical invariants
which were introduced
in \cite{Tom06} to study exotic examples of simple amenable
C*-algebras that are not $\mathcal{Z}$-stable. Sometimes there is a tight relationship between them. For example,
it was shown in Corollary~4.3 of \cite{Tom06} that
 $\rc(A)=0$ if and only if $\drr (A)=0$ whenever $A$ is  a simple infinite-dimensional real rank zero unital AH algebra.

The  Cuntz semigroup plays a crucial role in the Elliott program for the classification of C*-algebras \cite{CE08, ET08, Tom08}. 
See~\cite{APT11, CEI, TT15} for many aspects of the Cuntz semigroup.
It is generally  complicated and  large.
For simple nuclear \ca{s},
the classifiable ones are those whose Cuntz semigroups
are easily  understood (Section~5 of \cite{APT11}).
With the near completion of the Elliott program,
nonclassifiable \ca{s} receive more attention (see \cite{AGP19, HP19, Ph16})
and the Cuntz semigroup is the main additional available invariant.

Let $X$ be a compact metric space.
 The covering
dimension of $X$ is denoted by $\dim (X)$ and
the cohomological dimension with
rational coefficients is denoted by
$\dim_{\mathbb{Q}} (X)$.
For many results about the Cuntz semigroup  of $C(X)$ when $\dim (X) \leq 3$,
we refer to  the work of Robert and Tikusis~\cite{RT11}.
In the commutative setting, it is well known that the radius of comparison of $C(X)$ is dominated by
$\frac{1}{2}\dim (X)$ \cite{BRTTW12, EN13}.
For this reason, comparison theory can be viewed as a non-commutative dimension theory.
 This fact and the main result of~\cite{T11} are part of our motivation to wonder about an upper bound for
 the radius of comparison of C*-algebras of the form $C(X, A)$,
where $A$ is a unital C*-algebra.
The following conjecture is proposed by N. Christopher Phillips.
\begin{cnj}\label{Phi_Conjecture}
Let $A$ be a  stably finite unital C*-algebra and let $X$ be a compact metric space. Then
\begin{equation} \label{Eq_Con_191116}
\rc (A) \leq\rc \left( C(X) \otimes A \right) \leq \frac{1}{2}\dim (X)  + \rc (A) +1.
\end{equation}
\end{cnj}
The purpose
of this paper is to give some preliminary results related to
Conjecture~\ref{Phi_Conjecture}.

In Section~\ref{SEc_App_3}, we prove that the left-hand side  of (\ref{Eq_Con_191116})
is true for any stably finite unital C*-algebra $A$ and any compact metric space $X$.
(See Proposition~\ref{Prp.20199.02.17}.)
The right-hand side  of (\ref{Eq_Con_191116}) is also true for the case  $A =M_n(C(Y ))$ for $n\in \N$
with technical hypotheses on
$X$ and $Y$. (See Proposition~\ref{Prp.LS.Cojj.special}.)
But there are some difficulties  to prove it
 for a general \ca{} $A$ and a metric space $X$.
 So we decide to show that
 Conjecture~\ref{Phi_Conjecture} is valid at least for some choices of $A$  and $X$.

As an important special case, if $A$ is a residually stably finite $\mathcal{Z}$-stable unital C*-algebra,
then it is  shown in Proposition~\ref{Prp_Z_st_r} that $\rc (C(X)\otimes A)=0$.
To prove this, we apply Proposition~3.2.4(ii) of \cite{BRTTW12}.
Therefore,
finiteness and residual stable finiteness of $C(X, A)$ is one of the
starting points.

Another special case comes when $\dim (X)=0$.
(See Corollary~\ref{RC_Dim_0} and the discussion after it.)
To deal with this, we apply Proposition~\ref{Prp_CXA_Cu_gener} which states a
 connection between Cuntz comparison over $C(X, A)$ and $\dim (X)$.

It is definitely known that $\rc (A)=0$ for a stably finite exact simple $\mathcal{Z}$-stable unital C*-algebra $A$.
(See Corollary~4.6 of \cite{Rdm7}.)
One might naively expect that $\rc (A)=0$ for all stably finite exact (not necessarily simple)
$\mathcal{Z}$-stable unital C*-algebras.
To show that simplicity of $A$ is necessary,
we   exhibit a class of stably finite exact $\mathcal{Z}$-stable unital C*-algebras with nonzero radius of comparison.
(See Remark~\ref{Rm_191218}.)

In Section \ref{Sec_Drr}, we prove that if
$X$ is a compact metric space and $A$ is
a unital AH algebra with large matrix sizes, then
the dimension-rank  ratios of $A$ and $C(X,  A)$
are related by
$\drr (A) = \drr \left(C(X)\otimes A)\right).$
(See Proposition~\ref{Prp_drr_stabl}.)
There is an example in which
 $A$ doesn't have large matrix sizes and the equality fails. We further give a class of AH algebras $A$ with
$\rc \left(C(X)\otimes A)\right) = \rc (A)$.
So, (\ref{Eq_Con_191116}) is also true if $A$ is chosen from this class of AH algebras.

Along the way, we show the radius of comparison of a unital AH
algebra  is dominated by one half of its dimension-rank ratio.
(See Lemma~\ref{rc_half_drr}.) This result is known, but we have
not found it in the literature.

 At the end, we give many examples of
infinite-dimensional stably finite simple unital C*-algebras~$A$ for which
the radii of comparison of $A$ and $C(X)\otimes A$ are the same.
\subsection*{Acknowledgments}\label{Sec_Ackn}
Most of this research was done while the second author was
a visiting scholar at the University of Oregon
during the period March 2018 to
September 2019. He is thankful to that institution
for its hospitality. This paper will be part of the second author's Ph.D.\   dissertation.

Both authors would like to thank N. Christopher Phillips and Nasser
Golestani for several productive discussions. The second author also
thanks Aaron Tikuisis for his helpful comments about
Proposition~\ref{Prp_Z_st_r} and
Nathanial Brown, Marius D\v{a}d\v{a}rlat, Narutaka Ozawa, 
and S{\o}ren Eilers for answering a question about Definition~\ref{DPT09.Def.2.2}.
 Finally, both authors are grateful to the referee for a careful reading of the manuscript,
 for giving very valuable comments, and 
in particular for suggesting Corollary~\ref{Cor_AH_Cons1} and Corollary~\ref{Cor_AH_Cons2}.
\section{Preliminaries}
In this section,
we gather for easy reference
some information on the Cuntz semigroup, the dimension-rank ratio of AH algebras,
and the radius of comparison.
\begin{ntn}\label{N_9408_StdNotation}
We use the following standard notation.
If $A$ is a \ca, or if $A = M_{\infty} (B)$
for a C*-algebra~$B$, we write $A_{+}$ for the
set of positive elements of $A$.
The unitization of a C*-algebra $A$ is denoted by $A^+$.
(We add a new unit even if $A$ is already unital.)
 We let $\cK$ denote the algebra of compact operators on a separable
and infinite-dimensional Hilbert space $\cH$.
\end{ntn}
\subsection{The Cuntz semigroup}
Let $A$ be a \ca.
For $a, b \in M_{\infty} (A)_{+}$,
we say that $a$ is {\emph{Cuntz subequivalent to~$b$ in~$A$}},
written $a \precsim_{A} b$,
if there is a sequence $(c_n)_{n = 1}^{\infty}$ in $M_{\infty} (A)$
such that
$
\limi{n} c_n b c_n^* = a.
$
We say that $a$ and $b$ are {\emph{Cuntz equivalent in~$A$}},
written $a \sim_{A} b$,
if $a \precsim_{A} b$ and $b \precsim_{A} a$.
This relation is an equivalence relation,
and we write $\langle a \rangle_A$ for the equivalence class of~$a$.
We define $\W (A) = M_{\infty} (A)_{+} / \sim_A$,
together with the commutative semigroup operation
$\langle a \rangle_A + \langle b \rangle_A
 = \langle a \oplus b \rangle_A$
and the partial order
$\langle a \rangle_A \leq \langle b \rangle_A$
if $a \precsim_{A} b$.
We write $0$ for~$\langle 0 \rangle_A$.
We also define  $\Cu (A)= \W (A \otimes \cK)$.

The common notation for Cuntz subequivalence is $a \precsim b$ and it is originally from~\cite{Cun78}.
We include $A$ in the notation
because we want  notation for  the  Cuntz subequivalence
with respect to C*-algebras.

Part (\ref{PhiB.Lem_18_4_11}) of the following
is taken from Proposition 2.4  of~\cite{Ror92}.
 Part~(\ref{PhiB.Lem_18_4_10}) is Lemma 2.5(ii) of ~\cite{KR00}.
 Part~(\ref{PhiB.Lem_18_4_10.XV}) is Lemma~1.5 of \cite{Ph14}.
\begin{lem}\label{PhiB.Lem_18_4}
Let $A$ be a \ca.
\begin{enumerate}
\item\label{PhiB.Lem_18_4_11}
Let $a, b \in A_{+}$.
Then \tfae:
\begin{enumerate}
\item\label{PhiB.Lem_18_4_11.a}
$a \precsim_A b$.
\item\label{PhiB.Lem_18_4_11.b}
$(a - \ep)_{+} \precsim_A b$ for all $\ep > 0$.
\item\label{PhiB.Lem_18_4_11.c}
For every $\ep > 0$ there is $\dt > 0$ such that
$(a - \ep)_{+} \precsim_A (b - \dt)_{+}$.
\end{enumerate}
\item\label{PhiB.Lem_18_4_10}
Let $a, b \in A_{+}$ and let $\ep > 0$.
If $\| a - b \| < \ep$, then
$(a - \ep)_{+} \precsim_A b$.
\item\label{PhiB.Lem_18_4_10.XV}
Let $a, b \in A_+$ and let $\ep_1, \ep_2 \geq 0$. Then
\[
((a + b) - (\ep_1 + \ep_2))_+
\sim_A
(a - \ep_1)_+ + (b - \ep_2)_+
\precsim_A
(a - \ep_1)_+ \oplus (b - \ep_2)_+.
\]
\end{enumerate}
\end{lem}
\subsection{The dimension-rank ratio of an AH algebra}
Part~(\ref{DPT09.Def.2.2.1}) and Part~(\ref{DPT09.Def.2.2.2})
of the following definition are taken from Definition~1.1 of \cite{Tom06}.
Part~(\ref{DPT09.Def.2.2.5}) is defined by us to simplify notation.
What we call strictly homogeneous here was called homogeneous  in
\cite{DPT09, Tom06}. 
We want to avoid conflict with Definition~IV.1.4.1 of \cite{BlEC}.
\begin{dfn}\label{DPT09.Def.2.2}
A C*-algebra is said to be \emph{strictly homogeneous} if it is isomorphic to
$p \big (C(X)\otimes \cK \big )p$ for a compact Hausdorff space $X$ and
a projection of constant rank $p \in C(X)\otimes \cK$.
A \emph{strictly semihomogeneous} \ca ~is
a finite direct sum of strictly homogeneous \ca{s}.
\begin{enumerate}
\item\label{DPT09.Def.2.2.1}
An \emph{approximately homogeneous} (AH) \ca{}
is a direct limit
$$A =\dirlim (A_{j}, \psi_j)$$
where  $A_j$ is strictly semihomogeneous for each $j$.
\item\label{DPT09.Def.2.2.2}
Let $A =\dirlim (A_{j}, \psi _j)$
be a unital (i.e., both $A_j$ and
$\psi_j  \colon A_j \to  A_{j+1}$
 are unital for every $ j \in  \N$)
AH algebra, where
$$A_j  = \bigoplus_{l=1}^{m_j}
 p_{j,l} \big ( C(X_{j,l})\otimes \cK  \big) p_{j,l}$$
for compact Hausdorff spaces $X_{j,l}$, projections
 $p_{j,l} \in  C(X_{j,l})\otimes \cK $,
  and natural numbers $m_j$. Define
$\psi_{jk} = \psi_{j-1} \circ
\psi_{j-2} \circ \cdots \circ \psi_k$,
and write $\psi_{j \infty} \colon A_j \to  A$
 for the canonical map. We consider
this collection of objects and maps as
a decomposition for $A$. If  $\psi_j$ is  injective for every $j$,
 then we describe this collection as an injective
decomposition.
\item\label{DPT09.Def.2.2.5}
If a unital  AH algebra $A$ admits a decomposition as in (\ref{DPT09.Def.2.2.2}) for which
\[
\lim_{j \to \infty} \min_{1 \leq l \leq m_j} \big( \rank (p_{j,l}) \big) = \infty,
\]
then we say that $A$
has \emph{large matrix sizes}.
\end{enumerate}
\end{dfn}
\begin{dfn}[Definition~2.1 of \cite{Tom06}]
Let $A$ be a unital AH algebra.
 The \emph{dimension-rank ratio} of $A$, denoted $\drr (A)$,
is the infimum of the set of strictly positive reals
 $r$ such that $A$ has a decomposition which satisfies
$$\limsup \limits_{j\to \infty} \max_{1 \leq l \leq m_j}
 \left( \frac{\dim (X_{j,l})}{\rank (p_{j,l})} \right)=r, $$
whenever this set is not empty, and $\infty$ otherwise.
\end{dfn}
 Part (\ref{Prp2.2.Tom06.a}), Part (\ref{Prp2.2.Tom06.b}), and Part (\ref{Prp2.2.Tom06.c}) of
 the following lemma are parts of Proposition~2.2 of \cite{Tom06} and Part (\ref{Prp2.2.Tom06.d}) is
 Corollary~2.4 of \cite{Tom06}.
\begin{lem} \label{Prp2.2.Tom06}
Let $A$ and  $B$ be unital AH algebras and let $I$ a closed ideal of $A$. Then:
\begin{enumerate}
\item \label{Prp2.2.Tom06.a}
 $\drr( A / I ) \leq \drr(A)$.
\item \label{Prp2.2.Tom06.b}
 $\drr( A  \oplus B ) = \max \big(\drr(A), \drr(B)\big)$.
\item \label{Prp2.2.Tom06.c}
$\drr (A \otimes M_k) \leq \frac{1}{k} \drr(A)$ for $k\in \N$.
\item \label{Prp2.2.Tom06.d}
Let $A = \dirlim  (A_j ,\varphi_j )$ be a unital AH algebra where each $A_j$ is semihomogeneous.
Then
$\drr (A) \leq  \liminf_{j \to \infty} \drr (A_j)$.
\end{enumerate}
\end{lem}
\subsection{Radius of comparison}
The set of normalized 2-quasitraces on a unital C*-algebra $A$ is denoted by $\QT (A)$. 
By the discussion after Proposition II.4.6 of~\cite{BH82},
we have  $\QT ( A )\neq \varnothing$
for every stably finite unital C*-algebra $A$.
See \cite{BH82} for more details about 2-quasitraces.
\begin{dfn}
[Definition~12.1.7 of~\cite{GKPT18}]
\label{D_9422_dtau_dfn}
Let $A$ be a unital \ca,
and let $\tau \in \QT (A)$.
Define $d_{\tau} \colon \Mi (A)_{+} \to [0, \infty)$
by
$d_{\tau} (a) = \lim_{n \to \infty} \tau (a^{1/n})$.
We also use the same notation for the corresponding functions
on $\Cu (A)$ and $\W (A)$.
\end{dfn}

The following is Definition~6.1 of~\cite{Tom06},
except that we allow $r = 0$ in~(\ref{rc_dfn.a}).
This change makes no difference.

\begin{dfn}\label{rc_dfn}
Let $A$ be a stably finite unital C*-algebra.

\begin{enumerate}
\item\label{rc_dfn.a}
Let $r \in [0, \I)$.
We say that $A$ has {\emph{$r$-comparison}} if whenever
$a, b \in M_{\infty} (A)_{+}$ satisfy
$d_{\ta} (a) + r < d_{\ta} (b)$
for all $\ta \in \QT (A)$,
then $a \precsim_A b$.
\item\label{rc_dfn.b}
The {\emph{radius of comparison}} of~$A$, denoted
${\operatorname{rc}} (A)$, is
\[
\rc (A)
 = \inf \big( \big\{ r \in [0, \I) \colon
    {\mbox{$A$ has $r$-comparison}} \big\} \big)
\]
if it exists, and $\infty$ otherwise.
\end{enumerate}
\end{dfn}

If $A$ is simple,
then the infimum in \Def{rc_dfn}(\ref{rc_dfn.b})
is attained,
that is, $A$ has ${\operatorname{rc}} (A)$-comparison.
(See Proposition~6.3 of~\cite{Tom06}.)
\begin{thm}[Corollary~1.2 of \cite{EN13}]
\label{Thm_El_Ni_C(X)}
Let $X$ be  a compact metrizable space with $\dim_{\mathbb{Q}} (X) = \dim (X)$. Then:
\begin{enumerate}
\item\label{Thm_El_Ni_C(X).a}
If $\dim (X)$ is even, then
$ \frac{\dim (X)}{2} -2 \leq \rc \big( C(X) \big) \leq \max \Big(0, \frac{\dim (X) }{2} -1\Big)$.
\item\label{Thm_El_Ni_C(X).b}
If $\dim (X)$ is odd, then
$ \rc \big( C(X) \big) = \max \Big(0,  \frac{\dim (X) -1}{2} -1\Big)$.
\end{enumerate}
\end{thm}
\subsection{Functionals}
Let $\Omega$ be a semigroup in the category $\mathbf{Cu}$ 
 which is given in Denfinition~4.1 of \cite{APT11}. A functional
on $\Omega$ is a map $\lambda \colon  \Omega \to [0, \infty]$ which is order preserving, additive, preserves suprema of
increasing sequences and satisfies $\lambda (0) = 0$. We denote the set of  functionals on $\Omega$ by  $\F(\Omega)$.
Let $A$ be  a unital \ca. The radius of comparison of $(\Cu (A)\, , \langle 1_A \rangle)$, denoted $r_A$, is
the infimum of the set of real numbers $r$ in $[0, \infty)$ such that, if whenever $a, b \in  M_{\infty}(A)_+$
satisfy $\lambda (\langle a \rangle_A) + r
\lambda (\langle 1_A \rangle_A) \leq \lambda (\langle b \rangle_A)$ for all $\lambda \in \F (\Cu (A))$,
then $a \precsim_A b$.

Since we allow
functionals on $\Cu(A)$ which take the value $\infty$ at the unit, it follows that
$r_A \leq \rc(A)$. For sufficiently finite C*-algebras, we have $r_A = \rc (A)$. 
\begin{dfn}
A  unital \ca{} is called \emph{residually stably finite} if
all of its quotients are stably finite.
\end{dfn}
\begin{prp}[Proposition~3.2.3 of \cite{BRTTW12}]
\label{BW.Prp.3.2.3}
Let $A$ be a residually stably finite unital \ca. Then $\rc (A) = r_A$.
\end{prp}
Part~(\ref{Prp_Quot.a}) of the following proposition is a suitable version  of Proposition~6.2(iii) of \cite{Tom06} and
Part~(\ref{Prp_Quot.b}) is a generalization of Corollary~5.4 of \cite{Tom09}.
\begin{prp}\label{Prp_Quot}
Let $A$ be a residually stably finite unital C*-algebra.
\begin{enumerate}
\item\label{Prp_Quot.a}
For every closed ideal $I$ of $A$, we have
$\rc \big(A/I\big) \leq \rc (A)$.
\item\label{Prp_Quot.b}
If $A = \dirlim A_j$
where the homomorphisms of the inductive system are unital, then
$\rc (A) \leq \liminf_{j\to \infty} \rc (A_j).$
\end{enumerate}
\end{prp}
\begin{proof}
Part (\ref{Prp_Quot.a}) is essentially immediate from
 residual stable finiteness of  $A/I$ and  Proposition~3.2.4(i) of \cite{BRTTW12}.
Also Part (\ref{Prp_Quot.b}) is  immediate from residual stable finiteness  of $A$, \Prp{BW.Prp.3.2.3}, and
 Proposition~3.2.4(iii) of \cite{BRTTW12}.
\end{proof}

The following remark is taken from Theorem 2.3 of \cite{Tom06} and Theorem 5.1 of \cite{Tom09}.
\begin{rmk}\label{Rmk_190903}
For every compact Hausdorff space $X$ and projection $p \in C(X) \otimes \cK$, we have
\[
\rc \Big( p (C(X) \otimes \cK) p \Big) \leq \max \left( 0, \frac{ \dim (X) - 1}{2 \, \rank (p)} \right)
\quad\text{ and } \quad
\drr \Big(p \big(C(X) \otimes \cK\big) p \Big) =  \frac{\dim (X)}{ \rank (p)}.
\]
\end{rmk}
\section{An approach to the radius of comparison of $C(X) \otimes A$}
\label{SEc_App_3}
In this section, we obtain a lower bound for the radius of comparison of $C(X)\otimes A$
when $A$ is any unital stably finite C*-algebra and $X$ is any compact metric space.

We further discuss the radius of comparison of $C(X) \otimes A$ when we have stronger assumptions on $A$ or $X$.

Let $A$ be a C*-algebra. We denote by $\Prim (A)$ the set of primitive ideals of $A$.
See Section~II.6.5 of \cite{BlEC}, Section~5.4 of \cite{Mur90}, and \cite{PP17} for more details about primitive ideals of a C*-algebra.
\begin{lem}
\label{Lem_191210}
Let $X$ be a compact metric space, let $A$ be a C*-algebra,
and let $J$ be a proper closed ideal of $C(X, A)$.
Then there is a closed subset $F \subseteq X \times \Prim (A)$ such that
\[
J = \big\{ f \in C(X, A) \colon f (x) \in I \text{ for all } (x, I) \in F \big\}.
\]
\end{lem}
\begin{proof}
The result is essentially immediate from
 Theorem~5.4.3 of \cite{Mur90}, Proposition~2.16 of  \cite{BK04} and Theorem~II.2.2.4 of  \cite{BlEC}.
\end{proof}
The first part of the following lemma is needed to apply  Definition~\ref{rc_dfn}
and the second part is needed  to prepare for Proposition~\ref{Prp_Z_st_r}.
\begin{lem}\label{Lem19.11.10}
Let $X$ be a compact metric space and let $A$ be a unital C*-algebra.
Then:
\begin{enumerate}
\item\label{Lem19.11.10.a}
If $A$ is stably finite, then $C(X, A)$ is also stably finite.
\item\label{Lem19.11.10.b}
If $A$ is residually stably finite, then $C(X, A)$ is also residually stably finite.
\end{enumerate}
\end{lem}

\begin{proof}
We  prove (\ref{Lem19.11.10.a}).
Since $C(X, A)$ is unital, we apply Lemma~5.1.2 of \cite{RLL00}.
Let $f \in  C(X, A)$ satisfy $f^* f = 1_{C(X, A)}$.
So, $f(x)^* f(x)= 1_A$ for all $x \in A$.
 Since $A$ is finite, it follows that $f f^* = 1_{C(X, A)}$.
A similar argument works if $f \in M_n (C(X, A))$ for $n\in \N$.

To prove (\ref{Lem19.11.10.b}), we must show that all quotients of $C(X, A)$ are stably finite.
Let $J$ be a proper closed ideal of $C(X, A)$.
Then, by Lemma~\ref{Lem_191210}, there is a closed subset $F \subseteq X \times \Prim (A)$  such that
\[
J =  \big\{ f \in C(X, A) \colon f (x) \in I \text{ for all } (x, I) \in F \big\}.
\]
Since $M_n \big(C(X, A)/ J\big) \cong  C(X,M_n(A))/ M_n (J)$ for all $n \in \N$,
 it suffices to show that $C(X, A)/ J$ is finite.
Let $f \in C(X, A)$ satisfy $f^* f + J= 1_{C(X, A)} +J$. Then $f^* f - 1_{C(X, A)} \in J$.
Thus $f(x)^* f(x) - 1_A \in I$ for all $(x, I) \in F$.
Since $A/ I$ is finite for all $(x, I) \in F$, it follows that
$f(x) f(x)^* - 1_A \in I$ for all $(x, I) \in F$.
This relation implies that $f f^* - 1_{C(X, A)} \in J$ and therefore $f f^* + J = 1_{C(X, A)} + J$.
\end{proof}
By Lemma~\ref{Lem19.11.10}(\ref{Lem19.11.10.a}) and  the discussion after Proposition II.4.6 of~\cite{BH82},
we have  $\QT (C(X) \otimes A )\neq \varnothing$
for every compact space $X$ and stably finite unital C*-algebra $A$.
\begin{lem} \label{Lem.3.2019.02.17}
Let $X$ be a compact metric space and $A$ be a unital \ca. Let
also $l \in \N$ and let $a,b \in M_{l}(A)_+$. Then $1_{C(X)} \otimes
a \precsim_{C(X) \otimes  A}  1_{C(X)} \otimes  b$ if and only if
$a \precsim_{A} b $.
\end{lem}
\begin{proof}
Without loss of generality, assume $l=1$. Set $B=C(X) \otimes  A$.

To show the backward implication,
assume $a \precsim_A b$.
Define a homomorphism $\varphi \colon A \to B$ by $\varphi(c)=1_{C(X)} \otimes c$.
 Using this and $a \precsim_A b$, we get  $1_{C(X)} \otimes a  \precsim_{B}  1_{C(X)} \otimes  b$.

To show the forward implication, assume $a\otimes 1_{C(X)} \precsim_{B} b\otimes 1_{C(X)}$.
It suffices to find a sequence $(v_n)^{\infty}_{n=1}$ in $A$ with
$\lim_{n \to \infty} \| v_n b v^*_n - a \|=0$.

Identifying $C(X) \otimes A$ with $C(X, A)$ in the standard way, we get
$a \precsim_{C(X, A)} b$.
This relation implies that there exists a sequence $(w_n)_{n=1}^{\infty}$ in
$C(X, A)$ such that
\[
\lim_{n \to \infty} \| w_n b w^*_n - a \|=0.
\]
Now fix some $x_0 \in X$ and define $v_n=w_n(x_0)$ for $n \in \N$.
\end{proof}
The following proposition is essentially  immediate from Proposition~\ref{Prp_Quot}(\ref{Prp_Quot.a})
as soon as we assume  that
$A$ is residualy stably finite.
\begin{prp}\label{Prp.20199.02.17}
Let $X$ be a compact metric space and
let $A$ be a stably finite unital \ca. Then
$\rc (A) \leq \rc \big( C(X) \otimes  A \big)$.
\end{prp}

\begin{proof}
Set $B=C(X) \otimes A$.
We may clearly assume $\rc (B) < \infty$.
Let $r \in [0, \I)$ and suppose that $B$ has $r$-comparison.
Let $l \in \N$. Let $a, b \in M_{l}(A)_+$ satisfy
\begin{equation} \label{Eq1_191127}
d_{\ta} (a) + r < d_{\ta} (b)
\end{equation}
for all $\tau \in \QT (A)$.
Let $\theta \in \QT(B)$.
Define
$\tau_{\theta} \colon A \to \mathbb{C}$ by
$\tau_{\theta} (a)= \theta \big(1_{C(X)} \otimes a \big)$ for all $a \in A$.
Then $\tau_{\theta} \in \QT(A)$ and
 $d_{\tau_{\theta}} (c)= d_{\theta} \big(1_{C(X)} \otimes c \big)$ for all $c\in M_{l}(A)_+$.
Using this and (\ref{Eq1_191127}), we get
$$d_{\theta} \left(1_{C(X)}\otimes a\right) + r < d_{\theta} \left(1_{C(X)} \otimes b\right)$$
for all $\theta \in \QT(B)$.
Since $B$ has $r$-comparison,
it follows that
$1_{C(X)} \otimes a \precsim_{B} 1_{C(X)} \otimes b$.
Then, by \Lem{Lem.3.2019.02.17}, $a \precsim_A b$.
Therefore
$\rc (A) \leq r$.
Taking the infimum over $r \in [0, \I)$
such that $B$ has $r$-comparison,
we get $\rc (A) \leq \rc \big( B \big)$.
\end{proof}
\begin{prp}\label{Prp.LS.Cojj.special}
Let $X$ and $Y$ be compact metric spaces with $\dim_{\Q}(X) = \dim (X)$,
$\dim_{\Q} (Y) = \dim (Y)$, and $\dim_{\Q} (X \times Y) = \dim (X \times Y)$. Then 
Conjecture~\ref{Phi_Conjecture} is true for $A= M_n(C(Y))$ for $n\in \N$.
\end{prp}
\begin{proof}
By Proposition~\ref{Prp.20199.02.17}, it suffices to show that the right-hand side of (\ref{Eq_Con_191116}) holds.
To prove it, we may assume $\dim (X \times Y)\geq 5$ and $\dim (Y) \geq 4$.
Then, by Theorem~\ref{Thm_El_Ni_C(X)}, we have the following:
\begin{enumerate}
\item
If $\dim (Y)$ is odd and $\dim (X \times Y)$ is even, then
\begin{align*}
\rc \left(C (X) \otimes A \right) - 1
& \leq \frac{1}{n} \left(\frac{\dim (X \times Y)}{2} - 1\right) - 1
\\
&\leq \frac{1}{2} \dim (X) + \frac{1}{n} \left(\frac{\dim (Y) - 1}{2}  - 1\right)
= \frac{1}{2} \dim (X) + \rc (A).
\end{align*}
\item
If $\dim (Y)$ is even and $\dim (X \times Y)$ is odd, then
\begin{align*}
\rc \left(C (X) \otimes A\right) - 1
&= \frac{1}{n} \left(\frac{\dim (X \times Y) -1}{2}  - 1\right) - 1
\\
&\leq \frac{1}{2} \dim (X) + \frac{1}{n} \left(\frac{\dim  (Y) }{2}  - 2 \right)
= \frac{1}{2} \dim (X) + \rc (A).
\end{align*}
\end{enumerate}
The argument above is similar for other choices of $\dim (Y)$ and $\dim (X \times Y)$.
\end{proof}
The following proposition provides a relationship between
pointwise Cuntz comparison over $C(X, A)$  and the topological
dimension of $X$.
Although we only apply the proposition to the case $\dim (X)=0$ in this paper, 
the more general setting of the proposition might be helpful
when dealing with pointwise Cuntz comparison over $C(X, A)$.
\begin{prp}\label{Prp_CXA_Cu_gener}
Let $X$ be a compact  metric space with $m=\dim (X) < \infty$. Let
$A$ be a unital \ca, let $l \in \N$, and let $a, b \in M_l(C(X, A))_+$.
 If $a(x) \precsim_A b(x)$ for all $x\in X$,
 then $a \precsim_{C(X, A)} 1_{M_{m+1}} \otimes b$.
\end{prp}

\begin{proof}
Set $B= C(X, A)$,  $\widetilde{b} = 1_{M_{m+1}} \otimes b$, and
$\Gamma_m =\{0, 1, \ldots, m\}$. We must show that $a \precsim_{B}
\widetilde{b}$. By \Lem{PhiB.Lem_18_4}(\ref{PhiB.Lem_18_4_11.b}),
it suffices  to show that for every $\ep > 0$ we have $(a -
\ep)_{+} \precsim_{B} \widetilde{b}$. So, let $\ep>0$ and let
$x\in X$ be arbitrary. We may assume without loss of generality
that $l=1$.

For $x\in X$, since $a(x) \precsim_A b(x)$,
there exists  $v_{x}$ in $A$ such that
\begin{equation}
\label{EQ.333tqz}
 \| v_{x} b(x) v^*_{x} - a(x)\|< \frac{\ep}{m+1}.
\end{equation}
Since $a, b \in B$, the map $\xi_{x} \colon X \to [0, \infty)$,
given by
\[
z \mapsto \|v_{x} b(z) v^*_{x} - a(z) \|,
\]
 is continuous.
Using this, (\ref{EQ.333tqz}), and  $\xi_{x} (x)<\frac{\ep}{m+1}$,
we get an open neighborhood $N(x)$ of $x$ such that, for all $z\in N(x)$,
 \begin{equation}\label{Eq191012t2qz}
 \left\|v_{x} b(z) v^*_{x} - a(z) \right\|< \frac{\ep}{m+1}.
\end{equation}
 Since $X$ is compact and $X=\bigcup_{x \in X} N(x)$, we can choose
$x_0, x_1, \ldots, x_n \in X$ such that $X=\bigcup_{l=0}^{n} N(x_l)$
for some $n \in \N$.
By Proposition~1.5 of \cite{KW04}, there exists an $m$-decomposable finite open refinement of
$\{N (x_l) \colon l = 0, \ldots, n\}$
of the form
\begin{equation}\label{EQ2.191017}
\{U_{0, j}\}_{j \in J_{0}} \cup \{U_{1,j}\}_{j \in J_{1}} \cup \ldots \cup \{U_{m,j}\}_{j \in J_{m}}
\end{equation}
with $U_{k, j} \cap U_{k, j'} = \varnothing$ for $j\neq j'$ and $k \in \Gamma_m$.

Choose a partition of unity subordinate to this cover, say
 $f_{k,j} \colon X \to [0, 1]$
such that, for all $x\in X$,
\begin{equation}\label{EQ5.191017}
\mathrm{supp}(f_{k,j}) \subseteq U_{k,j}
\quad
\text{ and }
\quad
\sum_{k\in \Gamma_m}\sum_{ j \in J_k} f_{k,j} (x) =1.
\end{equation}
Also, for all $k \in \Gamma_m$ and $j \in J_k$, we  choose  $l(k, j) \in \{0, 1, \ldots, n\}$
such that $U_{k, j} \subseteq N\left(x_{l(k, j)}\right)$.

Now set
\[
v = \diag \left( \sum_{j \in J_0} f^{1/2}_{0, j} v_{x_{l(0, j)} },
\sum_{j \in J_1} f^{1/2}_{1, j} v_{x_{l(1, j)} },
 \ldots,
 \sum_{j \in J_m} f^{1/2}_{m, j} v_{x_{l(m, j)} }\right).
\]
Clearly  $v \in  M_{m+1}(B)$.
Now we claim
\[
\left\| \diag \left(\sum_{j \in J_0} f_{0,j} a,
 \ldots, \sum_{j \in J_m} f_{m, j} a \right) -  v
\widetilde{b} v^*\right\| < \frac{\ep}{m+1}.
\]
To prove the claim,
for every $z\in X$, define
$\Lambda_z = \big\{ (s, t) \in \Gamma_m \times \big(\bigcup_{k=0}^{m} J_k\big)  \colon z \in U_{s, t} \big\}$.
Using (\ref{EQ2.191017}) and (\ref{EQ5.191017}) at the last step, we compute
\begin{align*}
v \widetilde{b} v^*
&=
\diag
 \left(
\sum_{j,t \in J_0} f^{1/2}_{0, j} f^{1/2}_{0, t} v_{x_{l(0, j)} } b v^*_{x_{l(0, t)}},
\ldots,
\sum_{j,t \in J_m} f^{1/2}_{m, j} f^{1/2}_{m, t} v_{x_{l(m, j)} } b v^*_{x_{l(m, t)} }
\right)
\\
&=
\diag
 \left(
\sum_{j \in J_0} f_{0, j}  v_{x_{l(0, j)}} b v^*_{x_{l(0, j)}},
\ldots,
\sum_{j \in J_m} f_{m, j}  v_{x_{l(m, j)} } b  v^*_{x_{l(m, j)} }
\right).
\end{align*}
Therefore,
 using this at the first step,
 using (\ref{Eq191012t2qz}) at the third step,
 and using the second part (\ref{EQ5.191017}) at the last step, for all $z\in X$,
\begin{align*}
&
\left\| \left[\diag \left(\sum_{j \in J_0} f_{0,j} a,
 \ldots, \sum_{j \in J_m} f_{m, j} a \right) -  v\widetilde{b} v^*\right](z)\right\|
\\
&=
\left\| \left[
\diag \left(
\sum_{j \in J_0} f_{0,j}\left(a - v_{x_{l(0, j)}} b v^*_{x_{l(0, j)}}\right)
,\ldots,
\sum_{j \in J_m} f_{m,j} \left(a - v_{x_{l(m, j)}} b v^*_{x_{l(m, j)}}\right)
\right) \right](z)
\right\|
\\
&\leq
\sum_{k=0}^{m} \sum_{j \in J_k} f_{k,j}(z) \left\|a(z) - v_{x_{l(k, j)}}(z) b(z) v^*_{x_{l(k, j)}}(z)\right\|
\\
&<
\sum_{(k, j)\in \Lambda_z } f_{k,j}(z)\cdot \frac{\ep}{m+1}
\\
&\leq
\sum_{k\in \Gamma_m} \sum_{j \in J_k} f_{k,j}(z) \cdot \frac{\ep}{m+1} = \frac{\ep}{m+1}.
\end{align*}
This completes the proof of the claim.
Using the claim and Lemma~\ref{PhiB.Lem_18_4}(\ref{PhiB.Lem_18_4_10}), we get
\begin{equation}\label{Eq1.191017}
\left(\diag \left(\sum_{j \in J_0} f_{0,j} a,
 \ldots, \sum_{j \in J_m} f_{m, j} a \right)- \frac{\ep}{m+1} \right)_+
\precsim_{B}  v \widetilde{b} v^* \precsim_{B} \widetilde{b}.
\end{equation}
Therefore,
using the second part of (\ref{EQ5.191017}) at the first step,
using Lemma~\ref{PhiB.Lem_18_4}(\ref{PhiB.Lem_18_4_10.XV}) at the second step,
and using (\ref{Eq1.191017}) at the last step,
\begin{align*}
(a- \ep)_+
&= \left( \left(\sum_{k \in \Gamma_m} \sum_{j \in J_k} f_{k,j} \right) a - \ep \right)_+
\\
&\precsim_B
\left(  \sum_{j \in J_0} f_{0,j} a - \frac{\ep}{m+1} \right)_+
\oplus
\ldots
\oplus
\left(  \sum_{j \in J_m} f_{m,j} a - \frac{\ep}{m+1} \right)_+
\\
&=
\left(\diag \left(\sum_{j \in J_0} f_{0,j} a,
 \ldots, \sum_{j \in J_m} f_{m, j} a \right)- \frac{\ep}{m+1} \right)_+
 \precsim_B
\widetilde{b}.
\end{align*}
This completes the proof.
\end{proof}
The following corollary is  essentially immediate from Proposition~\ref{Prp_CXA_Cu_gener}.
\begin{cor}\label{Cor_CXA_Cu_Dim}
Let $X$ be a compact  metric space with $\dim (X)=0$. Let
$A$ be a unital \ca, let $l \in \N$, and let $a, b \in M_l(C(X, A))_+$.
Then  $a \precsim_{C(X, A)} b$ if and only if $a(x) \precsim_A b(x)$ for all $x\in X$.
\end{cor}
As an application of Corollary~\ref{Cor_CXA_Cu_Dim}, we
obtain the following corollary.
\begin{cor}\label{RC_Dim_0}
Let $X$ be a compact  metric space with $\dim (X)=0$ and
 let $A$ be a stably finite unital C*-algebra.
 Then $\rc (C(X)\otimes A) = \rc (A)$.
\end{cor}
\begin{proof}
Set $B=C(X, A)$.
By Proposition~\ref{Prp.20199.02.17}, it suffices to show $\rc (B) \leq \rc (A)$.
 We may assume $\rc(A) < \infty$.
 Let $r \in [0, \I)$ and suppose that $A$ has $r$-comparison.
Let $l \in \N$ and  let $a, b \in M_{l}(B)_+$.
Assume
\begin{equation}\label{EQ191012t1}
d_{\tau} (a) + r < d_{\tau} (b)
\end{equation}
 for all $\tau \in \QT (B)$.
  Let $\rho \in \QT (A)$. Then $\rho \circ \mathrm{ev}_x  \in \QT (B)$
 for all $x\in X$.
Therefore
$d_{\rho \circ \mathrm{ev}_x} (c) = d_{\rho} (c(x))$
for all $x\in X$, all $\rho \in \QT (A)$, and all $c \in B$.
Using this and (\ref{EQ191012t1}), we get
\[
d_{\rho} (a(x)) + r < d_{\rho} (b(x))
\]
for all $x\in X$ and all $\rho \in \QT(A)$.
Since $A$ has $r$-comparison, it follows that $a(x) \precsim_A b(x)$ for all $x \in X$.
 Then, by Corollary~\ref{Cor_CXA_Cu_Dim},
$a \precsim_{B} b$.
Therefore $\rc (B) \leq r$. Taking the infimum over $r \in [0, \I)$
such that $A$ has $r$-comparison,
we get $\rc (B) \leq \rc(A)$.
\end{proof}
Note that if the C*-algebra $A$ in Corollary \ref{RC_Dim_0} is
also  residually stably finite, then the result is  immediate from
Proposition~3.2.4(i) of \cite{BRTTW12},
Proposition~\ref{Prp.20199.02.17}, and the fact that $X= \invlim
X_j$ with $X_j$ a finite set for all $j$.

Corollary~\ref{RC_Dim_0} fails if $\dim (X) >0$.
For example, it is enough  to take $A=\mathbb{C}$ and
$X=[0, 1]^5$.

The following proposition is  a consequence of Proposition~3.2.4(ii) of \cite{BRTTW12}.
\begin{prp}\label{Prp_Z_st_r}
Let $X$ be a compact  metric space and  let $A$ be a residually stably finite $\mathcal{Z}$-stable unital C*-algebra.
 Then $\rc (C(X)\otimes A)=0$.
\end{prp}

\begin{proof}
First of all, although  simplicity of $A$ was assumed in Proposition~3.2.4(ii) of \cite{BRTTW12},
it is not necessary for the backward implication.
That is, almost unperforation implies $r_A=0$ even when $A$ is not simple.

On the other hand,  $\Cu(C(X, A))$ is almost unperforated by
Theorem~4.5 of \cite{Rdm7}. Thus, $r_{C(X, A)}=0$.

Since $A$ is residually stably finite,
 it follows from Lemma~\ref{Lem19.11.10}(\ref{Lem19.11.10.b}) that
$C(X, A)$ is also residually stably finite.
Therefore, using Proposition~\ref{BW.Prp.3.2.3}, we get $\rc(C(X, A))=r_{C(X, A)}=0$.
\end{proof}
Combining Proposition~\ref{Prp.20199.02.17} and Proposition~\ref{Prp_Z_st_r}, we get the following corollary which is a generalization of Corollary~4.6 of \cite{Rdm7}. 
\begin{cor}\label{Cor_Z_st_r}
Let $A$ be a residually stably finite $\mathcal{Z}$-stable unital C*-algebra.
 Then $\rc (A)=0$.
\end{cor}
In light of Proposition~\ref{Prp_Z_st_r} and Corollary~\ref{Cor_Z_st_r}, it is an important question
whether there exists a stably finite $\mathcal{Z}$-stable unital C*-algebra
with nonzero radius of comparison.
Such a C*-algebra can't be residually stably finite. We will give examples in the next section.
%
\section{A class of stably finite exact $\mathcal{Z}$-stable unital C*-algebra with nonzero radius of comparison}
It is  a result of R{\o}rdam \cite{Rdm7} that all stably finite exact simple $\mathcal{Z}$-stable  unital C*-algebras
 have strict comparison of positive elements.
To see that not all  $\mathcal{Z}$-stable unital C*-algebras
have strict comparison of positive elements, we exhibit
a class of stably finite exact $\mathcal{Z}$-stable unital C*-algebras
 with nonzero radius of comparison.
\subsection{Cones over C*-algebras}
Let $A$ be a C*-algebra. The
\emph{cone} over $A$, denoted $\Cone A$, is the set of
continuous functions $f \colon [0, 1] \to A$ with $f(0) = 0$.
 Clearly $\Cone A \cong C_0((0, 1]) \otimes A$.
If $A$ is unital, $(\Cone A)^+$ is isomorphic to  the
set of continuous functions $f \colon [0, 1] \to A$ such that $f(0)\in \mathbb{C} 1_A$.
For a nonunital C*-algebra $A$ we say that $A$ is stably finite
if $A^+$ is stably finite.
Then the cone over any C*-algebra is stably finite but often fails to be
residually stably  finite.
In general, we do not know whether
the tensor product of two (non-exact) stably finite simple C*-algebras is stably finite or not.
But we have the following nice lemma
for the case that one tensor factor is the unitization of the cone over a unital C*-algebra.
\begin{lem}
Let $A$ be a unital C*-algebra and let $B$ be a stably finite unital C*-algebra.
Then $(\Cone A)^+ \otimes B$ is stably finite for any choice of tensor product.
\end{lem}
\begin{proof}
Let $k\in \N$ and let $f \in  (\Cone A)^+ \otimes M_k(B)$.
Suppose $f^* f= 1_{(\Cone A)^+} \otimes 1_{M_{k} (B)}$. We must show $f f^*=  1_{(\Cone A)^+} \otimes 1_{M_{k} (B)}$.
We may identify $f$ as a continuous function $f \colon [0, 1] \to A \otimes M_{k} (B)$
with $f(0)= 1_{A} \otimes b$ for some $b \in M_k (B)$.
So, it suffices to show that $f(x) f(x)^*= 1_{A} \otimes 1_{M_{k} (B)}$ for all $x \in [0, 1]$.

We note that
\[
1_A \otimes (b^* b) = f(0)^* f(0)=1_{A} \otimes 1_{M_{k} (B)}.
\]
This relation implies that $b^* b= 1_{M_{k} (B)}$. Since $B$ is stably finite, it follows that
$b b^*= 1_{M_{k} (B)}$. Therefore
\[
f(0) f(0)^*= 1_A \otimes (b b^*)= 1_{A} \otimes 1_{M_{k} (B)}.
\]
Now define a uniformly continuous map $\Delta \colon [0, 1] \to A \otimes M_{k}(B)$ by $\Delta  (x) =f(x) f(x)^*$.
 Clearly $\Delta (0)= 1_{A} \otimes 1_{M_{k} (B)}$ and
 $\Delta (x) \in \Proj \big(A \otimes M_{k}(B)\big)$ for all $x \in [0, 1]$.
Choose $\delta > 0$ such that
$\|\Delta (t) - \Delta (s)\| < 1$ whenever $|t - s|<\delta$ for $s, t \in [0, 1]$.
Choose $n \in \N$ and
the partition $\mathcal{P}= \{t_0=0,t_1, \ldots, t_n=1\}$ of $[0, 1]$ with
$\max_{1\leq j\leq n} |t_{j} - t_{j-1}|<\delta$.
 Then $\|\Delta (t_j) - \Delta (t_{j-1})\| < 1$ for $j = 1, \ldots, n$.
Using this, $\Delta (t_j)$ is projection for $j = 0, 1, \ldots, n$, and Lemma~10.1.7 of \cite{GKPT18},
 we get  $\Delta (x)= 1_{A} \otimes 1_{M_{k} (B)}$ for all $x \in [0, 1]$
 and therefore   $f^*(x) f(x)= 1_{A} \otimes 1_{M_{k} (B)}$.
\end{proof}
We let $\T (A)$ denote the set of normalized traces on a unital C*-algebra.
 To prepare the next proposition, we need the following remark.
\begin{rmk}\label{Rmk_191126}
Let $A$ be unital C*-algebra with $\tau|_{\Cone A}= 0$ for all $\tau \in \T ((\Cone A)^+)$.
We may identify $(\Cone A)^+$
as the set of continuous functions  $f \colon [0, 1] \to A$ such that $f(0)\in \mathbb{C} 1_{A}$.
Define $\rho \colon (\Cone A)^+ \to \mathbb{C}$ by
$\rho (f) = \lambda_f$ where $f(0)= \lambda_f 1_{A}$.
Clearly $\rho \in \T \big((\Cone A)^+\big)$. Now we claim that $\rho$ is the  only tracial state on $(\Cone A)^+$.

To prove the claim, let $\tau \colon (\Cone A)^+ \to \mathbb{C}$
be  a tracial state and $f\in (\Cone A)^+$. Since $\tau |_{\Cone
A} = 0$, we have
\[
\tau (f)= \tau \big(f - f(0) 1_{A}\big) + \tau \big(f(0) 1_{A}\big)= \rho (f).
\]
Hence, $\rho$ is the unique trace on $(\Cone A)^+$. Moreover,  let
$f \in M_l ((\Cone A)^+)_+$ for some $l \in \N$. Choose $f_{jk}
\in (\Cone A)^+$ for $j, k=1,\ldots, l$ with $f=
\big(f_{jk}\big)_{l\times l}$. Since $f_{jk} (0) \in \mathbb{C}
1_{A}$ for $j, k=1,\ldots, l$, it follows that
$f(0)=\big(\lambda_{jk} 1_A\big)_{l\times l}$. Set $\lambda_f =
\big(\lambda_{jk} \big)_{l\times l}$. Now we have
\begin{equation}\label{Eq1_191125}
d_{\Tr_l \otimes \rho} (f) = \lim_{n \to \infty} \left(\Tr_l \otimes \rho\right) (f^{1/n})
= \lim_{n \to \infty} \Tr_l (\lambda_{f}^{1/n})=
\rank (\lambda_f).
\end{equation}
\end{rmk}
It is known (see Theorem~5.11 of~\cite{Hag14}) that all 2-quasitraces on a unital exact \ca{} are traces.
\begin{prp} \label{Prp.C.O.Z}
Let $A$ be a  exact unital C*-algebra such that $\tau|_{\Cone A}$ is zero for all $\tau \in \T ((\Cone A)^+)$.
Let $B$ be a  stably finite exact unital C*-algebra.
Then
\[
\rc \big((\Cone A)^+\big)=
\rc \big((\Cone A)^+ \otimes_{\mathrm{min}} B\big)=\infty.
\]
\end{prp}
\begin{proof}
First we prove $\rc \big((\Cone A)^+\big)= \infty$. For every $r\in [0, \infty)$,
we must find $f, g \in M_{\infty} ((\Cone A)^+)_+$ such that
$
d_{\rho} (f) + r < d_{\rho} (g)$ and  $f \npreceq g$
where $\rho$ is as in Remark~\ref{Rmk_191126}.

 The largest integer which is less than or equal to $r$ is denoted by $\lfloor r \rfloor$.
Set $n= \lfloor r \rfloor + 2$.
Define $f_0, g_0 \colon [0, 1] \to [0, \infty)$ by
\[
f_0 (t)= 1
\qquad \text{ and } \qquad
g_0 (t)=
\begin{cases}
- 4 t +2 & 0 \leq t\leq 1/2\\
0  & 1/2 \leq t\leq 1.
\end{cases}
\]
We let $e_{j, k} \in M_n$ be the standard matrix units.
Then set $f(t)= f_0 (t) 1_{A} \otimes e_{1, 1}$
 and $g(t)= g_0 (t) 1_{A} \otimes 1_{M_n}$ for $t \in [0, 1]$.
Clearly $f, g \in M_n ((\Cone A)^+)_+$. By (\ref{Eq1_191125}), we have
$d_{\rho} (f)=1$ and $d_{\rho} (g)=n$. Therefore
\[
d_{\rho} (f) + r = 1 + r < \lfloor r \rfloor + 2 = n  = d_{\rho} (g).
\]
Now assume $f \precsim_{(\Cone A)^+} g$. Choose $v \in M_n ((\Cone A)^+)$ such that
$\| f - v g v^* \|< \tfrac{1}{2}$.
Then
\[
1=  \left\| f(1) - v(1) g(1) v^*(1)  \right\|\leq \| f - v g v^* \|< \tfrac{1}{2}.
\]
This is a contradiction.

Now we show $\rc \big((\Cone A)^+ \otimes_{\mathrm{min}} B\big) = \infty$.
The proof is similar to the first part,
except that we now find
 $\widetilde{f}$ and $\widetilde{g}$ instead of $f$ and $g$.
Since $(\Cone A)^+$ has a unique tracial state $\rho$,
it is easy to check that
\begin{equation} \label{EQ1.191213}
\T \big((\Cone A)^+ \otimes_{\mathrm{min}} B\big) =\big\{ \rho \otimes \tau \colon \tau \in  \T (B) \big\}.
\end{equation}
Set $\widetilde{f}= f \otimes 1_B$ and $\widetilde{g}= g \otimes 1_B$.
Clearly $\widetilde{f}, \widetilde{g} \in \big( M_{n} ((C A)^+) \otimes_{\mathrm{min}} B \big)_+$.
Using (\ref{Eq1_191125}) and (\ref{EQ1.191213}),
we get, for all $\theta \in \T \big((\Cone A)^+ \otimes_{\mathrm{min}} B\big)$,
\[
d_{\theta} (\widetilde{f})= d_{\rho} (f)=1
\quad \text{and}  \quad
d_{\theta} (\widetilde{g})=d_{\rho} (g)=n.
\]
Then a similar calculation gives us a similar contradiction as in the first part.
\end{proof}
We refer to Section~4 of \cite{KR00} for the definition of a  purely infinite C*-algebra and its properties.
\begin{cor}\label{Cor.C.O.Z}
Let $A$ be a purely infinite exact simple  unital C*-algebra and
$B$ be a stably finite exact unital C*-algebra. Then
\[
\rc \big((\Cone A)^+\big)=
\rc \big((\Cone A)^+ \otimes_{\mathrm{min}} B\big)=\infty.
\]
\end{cor}
\begin{rmk} \label{Rm_191218}
 Applying Corollary~\ref{Cor.C.O.Z} with $\mathcal{Z}$ in place of $B$, we get
$\rc ((\Cone A)^+ \otimes \mathcal{Z} )= \infty$. More generally,
if the C*-algebra $B$ in Corollary~\ref{Cor.C.O.Z} is also  $\mathcal{Z}$-stable,
then we can get a class of stably finite exact $\mathcal{Z}$-stable unital C*-algebras with nonzero radius of comparison.
\end{rmk}
\section{A class of AH algebras with particular radius of comparison}
\label{Sec_Drr}
In this section, we give a class of AH algebras satisfying Conjecture~\ref{Phi_Conjecture}.
 As preparation, we prove that $\drr (A) = \drr \big(C(X) \otimes A \big)$
for any unital AH algebra $A$ with large matrix sizes.

The dimension-rank ratio of an AH algebra is related to its radius of comparison by the following lemma.
\begin{lem}\label{rc_half_drr}
Let $A$
be a unital AH algebra.
Then
  $\rc (A) \leq \frac{1}{2} \drr (A)$.
\end{lem}

\begin{proof}
We may assume $\drr(A) < \I$.
Suppose $\dirlim A_k$ is an arbitrary decomposition for $A$, where
\[
A_k  = \bigoplus_{l=1}^{m_k}
 p_{k,l} \big ( C(X_{k,l})\otimes \cK  \big) p_{k,l}
 \]
for compact Hausdorff spaces $X_{k,l}$, projections
 $p_{k,l} \in  C(X_{k,l})\otimes \cK $,
  and  $m_k \in \N$. Set
\begin{equation}
\label{Eq3.190930}
 r = \limsup \limits_{k\to \infty} \max_{1 \leq l \leq m_k}
 \left ( \frac{\dim (X_{k,l}) }{\rank (p_{k,l})}\right )
\end{equation}
which is a nonnegative real number or $\infty$.
Using Proposition~6.2(i) of \cite{Tom06} at the first step and
 \Rmk{Rmk_190903} at the second step, we get, for every $k$,
\begin{equation}\label{Eq7.190903}
\rc (A_k)
=
 \max_{ 1\leq l\leq m_k} \rc \left(p_{k,l} \big ( C(X_{k,l})\otimes \cK  \big) p_{k,l}\right)
 \leq
 \max_{ 1\leq l\leq m_k} \left( \frac{\dim (X_{k,l}) }{2 \,\rank (p_{k,l})} \right).
\end{equation}
It follows easily from the discussion after Definition~V.2.1.9 of \cite{BlEC} that
 AH algebras are residually stably finite.
Therefore, using Proposition~\ref{Prp_Quot}(\ref{Prp_Quot.b}) at the first step,
 using (\ref{Eq7.190903}) at the second step,
 and using (\ref{Eq3.190930}) at the last step,
\begin{align*}
\rc (A )
&\leq
\liminf_{j \to \infty} \rc (  A_{j} )
\leq
\liminf_{j \to \infty} \max_{ 1\leq l\leq m_j} \left( \frac{\dim (X_{j,l}) }{2 \,\rank (p_{j,l})} \right)
\\\notag
&\leq
\frac{1}{2} \limsup_{j \to \infty} \max_{ 1\leq l\leq m_j} \left( \frac{\dim (X_{j,l}) }{\rank (p_{j,l})} \right)
=  \frac{r}{2}.
\end{align*}
Using this and taking infimum over all $r$ with
$r= \limsup \limits_{k\to \infty} \max_{1 \leq l \leq m_k}
 \left ( \frac{\dim (X_{k,l}) }{\rank (p_{k,l})}\right )$, we get
$
 \rc (A) \leq \frac{1}{2} \drr (A).
$
\end{proof}
There is an example in which
 $A$ doesn't have large matrix sizes and the equality
$\rc (A) = \frac{1}{2} \drr (A)$
fails. For example, it is enough to take
 $A= C([0, 1]^6)$.

We now give a more precise statement of Proposition~6.8 of \cite{Tom06}.
\begin{cor}\label{Cor_AH_Cons1}
For every $r \in [0, \infty)$, there exists a simple unital AH algebra such that $\rc (A)=\frac{1}{2} \drr (A)=r$.
\end{cor}
\begin{proof}
Combine Proposition~6.8 of \cite{Tom06} and Lemma \ref{rc_half_drr}.
\end{proof}
\begin{prp}\label{Prp_drr_stabl}
Let $A$
be a unital AH algebra with large matrix sizes and let
  $X$ be a compact metric space.
  Then
  $\drr (A) = \drr \big(C(X) \otimes A \big)$.
\end{prp}

\begin{proof}
First we show
$ \drr \big(C(X) \otimes A \big) \leq \drr (A)$.
We may assume $\drr (A) < \infty$ and  $X= \invlim X_j$ with
$\dim (X_j) < \infty$ for all $j$ (by  Corollary~5.2.6 of \cite{SkiK}).
Suppose $\dirlim A_k$ is an arbitrary decomposition for $A$, where
\[
A_k  = \bigoplus_{l=1}^{m_k}
 p_{k,l} \big ( C(Y_{k,l})\otimes \cK  \big) p_{k,l}
 \]
for compact Hausdorff spaces $Y_{k,l}$, projections
 $p_{k,l} \in  C(Y_{k,l})\otimes \cK $,
  and  $m_k \in \N$. Set
\begin{equation}
 r = \limsup \limits_{k\to \infty} \max_{1 \leq l \leq m_k}
 \left ( \frac{\dim (Y_{k,l}) }{\rank ( p_{k,l}) }\right )
\end{equation}
which is a nonnegative real number or $\infty$.
Choose a strictly increasing sequence $r(j)$ such that, for all $j$,
\begin{equation}\label{Eq1.19.09.3}
\min_{1 \leq l \leq m_{r(j)} } \left( \rank (p_{r(j),l}) \right)\geq j \dim (X_j).
\end{equation}
Clearly, $C(X) \otimes A  \cong \dirlim C(X_j)  \otimes A_{r(j)} $.
 Using Lemma~\ref{Prp2.2.Tom06}(\ref{Prp2.2.Tom06.b}) at the second step,
using \Rmk{Rmk_190903} at the third step,
and using
\[
\dim \big( X_j \times Y_{r(j),l}\big)  \leq \dim (Y_{r(j),l}) +  \dim (X_j),
\]
the fourth step,
 we get, for every $j$,
\begin{align*}
&
\drr \left( C(X_j) \otimes A_{r(j)} \right)
\\
&\hspace*{2em} {\mbox{}} =
\drr \left(\bigoplus_{l=1}^{m_{r(j)}}
 (1_{C(X_j)} \otimes  p_{r(j),l}  ) \big ( C( X_j \times Y_{r(j),l})\otimes \cK  \big) (1_{C(X_j)} \otimes p_{r(j),l})  \right)
 \\
&\hspace*{2em} {\mbox{}} \leq
 \max_{1 \leq l \leq m_{r(j)}}
\left(
 \drr \Big(
 (1_{C(X_j)} \otimes p_{r(j),l}) \left( C( X_j \times Y_{r(j),l})\otimes \cK  \right) (1_{C(X_j)} \otimes  p_{r(j),l})\Big)
 \right)
\\
&\hspace*{2em} {\mbox{}} \leq
 \max_{1 \leq l \leq m_{r(j)} } \left( \frac{\dim  \left(X_j \times Y_{r(j),l}\right)}{\rank (p_{r(j),l})} \right)
\\
&\hspace*{2em} {\mbox{}} \leq
\max_{1 \leq l \leq m_{r(j)} }  \left(\frac{\dim (X_j) +  \dim (Y_{r(j),l})  }{ \rank (p_{r(j),l})} \right)
\\
&\hspace*{2em} {\mbox{}} \leq
\max_{1 \leq l \leq m_{r(j)}} \left( \frac{\dim (X_j) }{ \rank (p_{r(j),l}) } \right)
+
\max_{1 \leq l \leq m_{r(j)} } \left( \frac{\dim (Y_{r(j),l})}{ \rank (p_{r(j),l})} \right).
\end{align*}
Therefore, using this relation at the second step,
 using Lemma~\ref{Prp2.2.Tom06}(\ref{Prp2.2.Tom06.d}) at the first step,
and (\ref{Eq1.19.09.3}) at the fourth step,
\begin{align*}
&\drr \big(C(X)\otimes A\big)
\\
&\hspace*{3em} {\mbox{}} \leq
\liminf_{j \to \infty} \drr \left(C(X_j) \otimes A_{r(j)}\right)
\\
&\hspace*{3em} {\mbox{}} \leq
\liminf_{j \to \infty}
\left(
\max_{1 \leq l \leq m_{r(j)}} \left( \frac{\dim (X_j) }{ \rank (p_{r(j),l} )} \right)
+
\max_{1 \leq l \leq m_{r(j)} } \left( \frac{\dim (Y_{r(j),l})}{ \rank (p_{r(j),l})} \right)
\right)
\\
&\hspace*{3em} {\mbox{}} \leq
\limsup_{j \to \infty}
\left(
\max_{1 \leq l \leq m_{r(j)}} \left( \frac{\dim( X_j) }{ \rank (p_{r(j),l}) } \right)
\right)
+
\limsup_{j \to \infty}
\left(
\max_{1 \leq l \leq m_{r(j)} } \left( \frac{\dim (Y_{r(j),l})}{ \rank (p_{r(j),l})} \right)
\right)
\\
&\hspace*{3em} {\mbox{}} \leq
\lim_{j \to \infty} \frac{1}{j}
+
\limsup_{j \to \infty} \left(\max_{1 \leq l \leq m_{r(j)} } \left( \frac{\dim (Y_{r(j),l})}{ \rank (p_{r(j),l})} \right) \right)
 = r.
\end{align*}
Using this and taking the  infimum over all $r$ with
$r= \limsup \limits_{k\to \infty} \max_{1 \leq l \leq m_k}
 \left ( \frac{\dim (Y_{k,l}) }{\rank (p_{k,l})}\right )$, we get
\[
 \drr \big(C (X) \otimes A \big) \leq \drr (A).
\]
To show
$\drr (A) \leq \drr \big(C(X) \otimes  A\big)$,
let $x_0 \in X$. Define a surjective homomorphism
$\varphi \colon C (X, A) \to A$
 by $\varphi (f)= f(x_0)$.
 Using this at the first step and using Lemma~\ref{Prp2.2.Tom06}(\ref{Prp2.2.Tom06.a}) at the second step,
  we get
\[
\drr (A) = \drr \left( C(X, A) / \textrm{ker} \varphi \right) \leq  \drr \left(C(X) \otimes  A\right).
\]
This completes the proof.
\end{proof}
There is an example in which
 $A$ doesn't have large matrix sizes and the equality
$\drr (A) = \drr \big(C(X)\otimes A\big)$ fails. For example, it is enough to take $A= C([0, 1])$ and $X=[0, 1]$.

As we  promised, we introduce a class of AH algebras $A$
with $\rc (C(X)\otimes  A) = \rc ( A )$ in the following theorem.
\begin{thm}\label{Thm_rc_drr}
Let $A$
be a unital AH algebra with large matrix sizes and let $X$ be a compact metric space.
Suppose $\rc (A) = \frac{1}{2}\drr (A)$.
 Then:
\begin{enumerate}
\item \label{L1_rc_drr.a}
$\rc ( C(X) \otimes A   ) = \frac{1}{2} \drr (C(X) \otimes A)$.
\item \label{L1_rc_drr.b}
$\rc (C(X)\otimes  A) = \rc ( A )$.
\end{enumerate}
\end{thm}

\begin{proof}
We prove (\ref{L1_rc_drr.a}). Using Proposition~\ref{Prp_drr_stabl} at the first step,
using Proposition~\ref{Prp.20199.02.17} at the third step, and using Lemma~\ref{rc_half_drr} at the last step,
we get
\[
\frac{1}{2} \drr (C(X) \otimes A)
=
\frac{1}{2} \drr ( A)
=
\rc (A)
\leq
\rc ( C(X) \otimes A ) \leq \frac{1}{2} \drr (C(X) \otimes A).
\]
This relation implies $\rc ( C(X) \otimes A ) = \frac{1}{2} \drr (C(X) \otimes A)$.

We prove (\ref{L1_rc_drr.b}). Using Part (\ref{L1_rc_drr.a}) at the first step and
using Proposition~\ref{Prp_drr_stabl} at the second step, we get
\[
\rc ( C(X) \otimes A )
=
\frac{1}{2} \drr (C(X) \otimes A)
=
\frac{1}{2} \drr ( A)
=
\rc (A).
\]
This completes the proof.
\end{proof}
One should ask whether there exists  an AH algebra satisfying the hypotheses of Theorem~\ref{Thm_rc_drr}.
To see that there is such an AH algebra,
we give  a simple unital AH algebra~$A$
with large matrix sizes, stable rank one, nonzero radius of comparison, and
$\rc(A)=\frac{1}{2} \drr(A)$.
The example was recently constructed in Section~6 of \cite{AGP19}.
\begin{eme}\label{L_rc_drr}
Let $X$ be a compact metric space and
let $A= \dirlim A_n$ be the AH algebra which was constructed in Section~6 of \cite{AGP19}.
Adopt the assumptions and notation of Section~6 of \cite{AGP19}.
Then:
\[
\rc (A) = \frac{1}{2} \drr (A), \quad
\rc ( C(X) \otimes A   ) = \frac{1}{2} \drr (C(X) \otimes A),
\quad \text{and} \quad
\rc (C(X)\otimes  A) = \rc ( A ).
\]
By Lemma~\ref{rc_half_drr} and Theorem~\ref{Thm_rc_drr},
it suffices to show  $\drr (A) \leq 2 \, \rc (A)$.
Using Lemma~\ref{Prp2.2.Tom06}(\ref{Prp2.2.Tom06.d}) at the first step,
using Lemma~\ref{Prp2.2.Tom06}(\ref{Prp2.2.Tom06.b}) and
\Rmk{Rmk_190903} at the third step, and using Theorem~6.15 of \cite{AGP19} at the last step,
we get
\begin{align*}
\drr (A)
 &\leq \liminf_{n \to \infty} \drr ( A_n)
=
\liminf_{n \to \infty}
\drr \left( M_{r(n)} \left(C(X_n) \oplus C(X_n)\right) \right)
\\
&\leq
\liminf_{n \to \infty} \frac{\dim (X_n)}{r(n)}
=
\lim_{n \to \infty}
\frac{2 s(n)}{r(n)} = 2 \, \rc (A),
\end{align*}
as desired.
\end{eme}
Corollary~\ref{Cor_AH_Cons1} provides more examples
to feed into Theorem~\ref{Thm_rc_drr}.
\begin{cor}\label{Cor_AH_Cons2}
For every $r \in [0, \infty)$ and for every compact metric space $X$, there exists a simple unital AH algebra such that 
$\rc (C(X)\otimes  A)=\rc (A)=r$.
\end{cor}
\section{Open problem}\label{Sec_Q}

The most obvious problem is whether
the right-hand side of
inequality (\ref{Eq_Con_191116}) holds for any stably finite unital C*-algebra $A$
and any compact metric space $X$.

\begin{qst}\label{Q_191215}
Let $A$ be a unital stably finite  \ca{} and let $X$ be a compact metric space.
Does it follow that
\[\rc \left( C(X) \otimes A  \right) \leq \rc (A) + \frac{1}{2}\dim (X) +1?
\]
\end{qst}

\end{document}